\documentclass[11pt]{article}

\usepackage{amsmath,amsthm,amssymb,graphicx}
\usepackage{graphicx,float}
\usepackage[font=small]{caption}
\usepackage{enumitem}

\usepackage[usenames,dvipsnames,svgnames,table]{xcolor}

\usepackage{fullpage}

\title{\textbf{Spatial structure of Sinai-Ruelle-Bowen measures}}
\author{N.~Chernov, A.~Korepanov \\
University of Alabama at Birmingham\\
1300 University Boulevard\\
Birmingham, AL 35294-1170\\
USA}
\date{June 23, 2014}

\theoremstyle{plain}
\newtheorem{lemma}{Lemma}[section]

\theoremstyle{plain}
\newtheorem{theorem}{Theorem}[section]

\theoremstyle{plain}

\theoremstyle{definition}

\theoremstyle{remark}
\newtheorem{remark}{Remark}[section]

\newcommand{\eps}{\varepsilon}
\def\bJ{\mathbf{J}}
\def\bE{\mathbf{E}}
\def\cD{\mathcal{D}}
\def\cX{\mathcal{X}}
\def\cM{\mathcal{M}}
\def\cF{\mathcal{F}}
\def\cG{\mathcal{G}}
\def\cZ{\mathcal{Z}}
\def\bF{\mathbf{F}}
\def\bE{\mathbf{E}}
\def\bp{\mathbf{p}}

\begin{document}

\maketitle

\begin{abstract}
Sinai-Ruelle-Bowen measures are the only physically observable
invariant measures for billiard dynamical systems under small
perturbations. These measures are singular, but as it was noted in
\cite{BCKL}, marginal distributions of spatial and angular
coordinates are absolutely continuous. We generalize these facts and
provide full mathematical proofs.

\end{abstract}

\section{Introduction}

This work is motivated by our earlier studies \cite{BCKL} of physically
observable properties of Sinai-Ruelle-Bowen (SRB) measures for 2D
periodic Lorentz gases (Sinai billiards) with finite horizon, under
small perturbations.

Sinai billiards and their perturbations have some rough
singularities, but on the other hand they are strongly hyperbolic
\cite{Ch1}. Precisely, their Lyapunov exponents are non-zero and
unstable vectors grow uniformly, at a rate $\geq C\Lambda^n$, where
$n$ is the collision counter and $\Lambda>1$ the so called
hyperbolicity constant. While Sinai billiards are equilibrium
systems preserving a smooth (Liouville) measure, their perturbations
are non-equilibrium systems whose natural invariant measures (steady
states) are usually singular -- they are SRB measures.

An SRB measure is an ergodic invariant probability measure with
absolutely continuous conditional distributions on unstable
manifolds. SRB measures are the only physically observable measures
because their basins of attraction have positive Lebesgue volume;
see \cite{Y} and \cite[Sect.~5.9]{H}. Perturbations of Sinai
billiards have unique SRB measures, which are mixing and Bernoulli
\cite{Ch1,Ch2}.

In \cite{BCKL} we studied particularly interesting perturbations of
Sinai billiards where the particle moved under a small constant
external field $\bE$ subject to a Gaussian thermostat that kept its
speed constant. For that model, the SRB measure was constructed long
ago \cite{CELS,CELSp}, and it was proved that the (global) current
$\bJ$ of the particle satisfied $\bJ = \sigma\bE+o(|\bE|)$, with the
conductivity $\sigma$ given by a standard Green-Kubo formula.

In that model, as well as in many other perturbations of Sinai
billiards, the SRB measure is singular with respect to the Lebesgue
measure (the Hausdorff dimension of the SRB measure is lower than
that of the phase space \cite{CELS}). It is also argued in
statistical mechanics that many multiparticle systems (gases and
fluids) under external forces develop nonequilibrium steady states
that behave as SRB measures, in particular they are singular with
respect to the phase volume (this is known as ``chaotic hypothesis''
or ``Axiom C'' \cite{GC}).

However one rarely observes the a steady state on the entire phase
space to see its singularity. Usually one observes selected
variables, such as positions or velocities of certain moving
particle(s). And computer simulations show that those selected
variables have surprisingly continuous distributions. In \cite{BCKL}
we showed that for the Lorentz gas model three selected variables --
the local particle density, the local current, and the angular
velocity distribution have continuous densities. We derived
Green-Kubo type formulas for those densities.

In \cite{BCKL} we only sketched the arguments and presented numerical
evidence supporting our conclusions. In this paper we provide full
mathematical proofs and generalize our conclusions to wider classes of
perturbations and selected variables.

\section{Model}
\label{sec:model}
Our work is an extension of papers \cite{Ch1} and \cite{Ch2}, and for
consistency we follow their definitions and notations whenever
possible. We also refer the reader to these papers for more details on
the model.

Let $\mathcal{D} = \mathbb{T}^2 \setminus \cup_{i=1}^k \mathcal{B}_i$
be a 2D torus without a finite union of disjoint open convex domains
$\mathcal{B}_i$ whose boundary is $C^3$ smooth and has non-vanishing
curvature. Two particular tables of that sort are shown in
Figure~\ref{pic:tables}.

\begin{figure}[h!]
\centering
\begin{tabular}{ c c c}
\includegraphics[clip,width=0.4\textwidth]{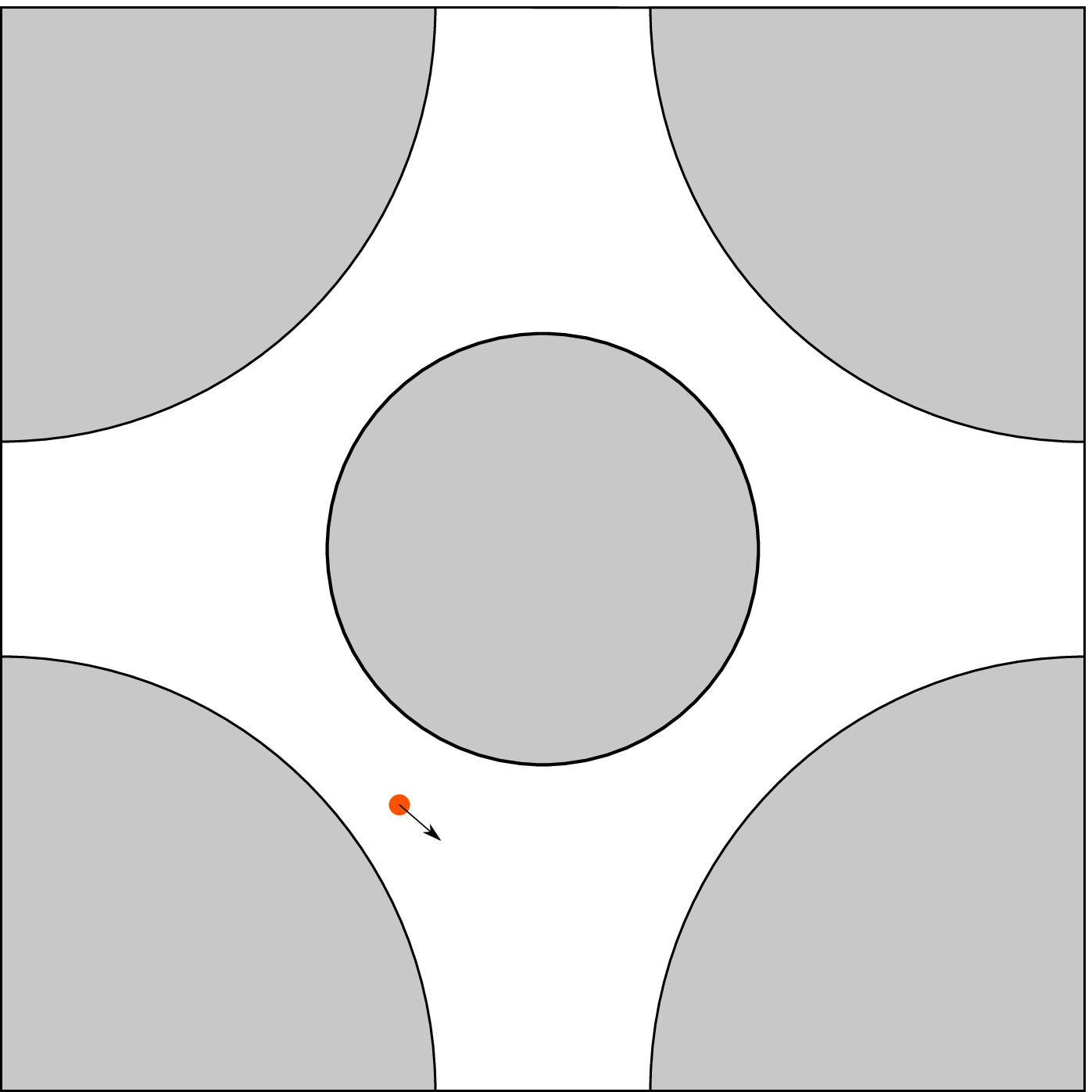} &
\includegraphics[clip,width=0.4\textwidth]{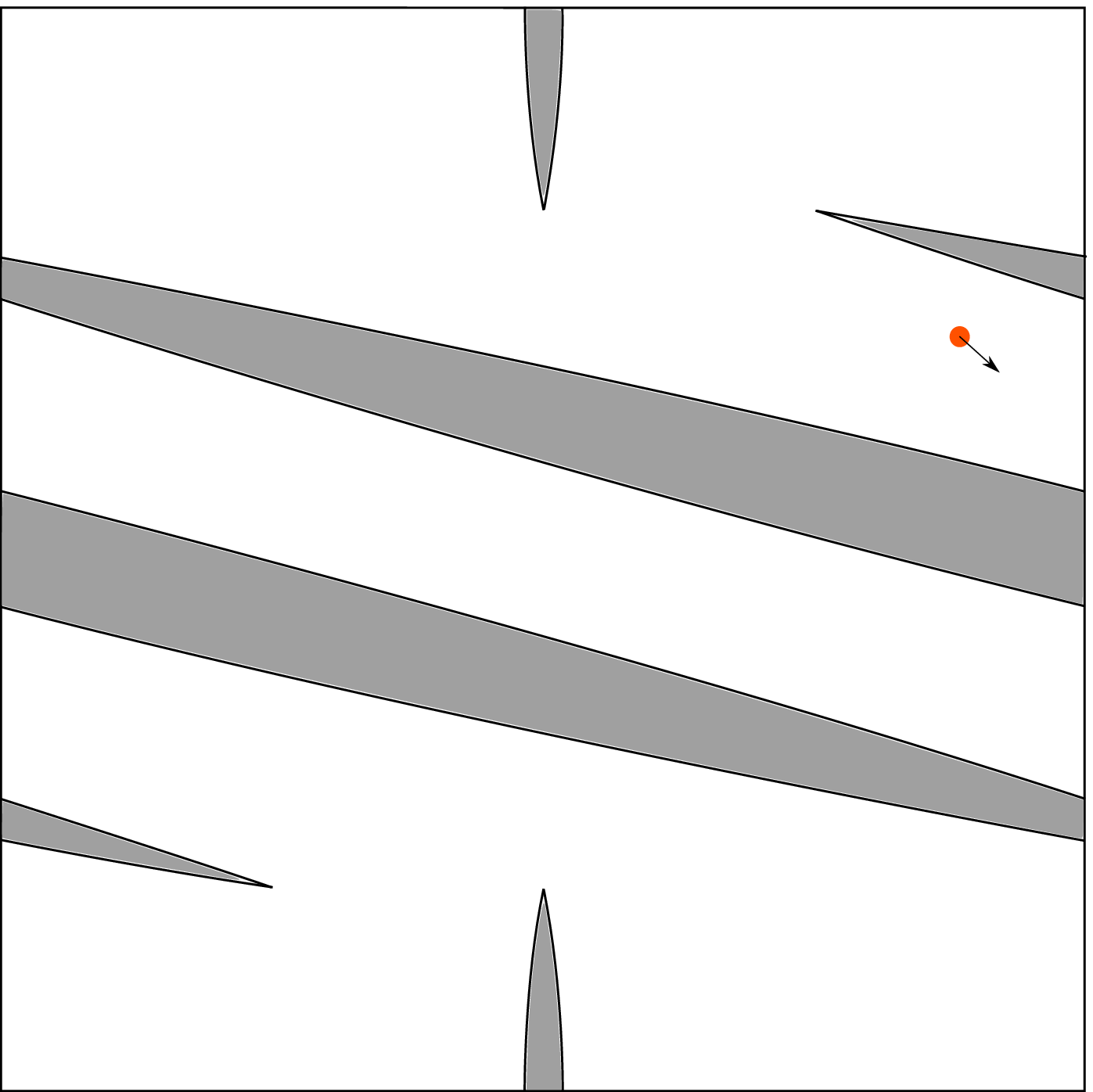} \\
A & B \\
\end{tabular}
\caption{Tables A and B used in our numerical experiments with Gaussian
Thermostat.\label{pic:tables}}
\end{figure}

A particle moves in $\mathcal{D}$ according to equations
\begin{equation}
\begin{cases}
\dot{\mathbf{q}}=\mathbf{p} \\
\dot{\mathbf{p}}=\mathbf{F}(\mathbf{p},\mathbf{q})
\end{cases}
\end{equation}
where $\mathbf{F}(\mathbf{p},\mathbf{q})=(F_1,F_2)$ is a stationary
(independent of time) force. Collisions with obstacles are elastic (the
speed is preserved) and specular (the angle of reflection is equal to
the angle of incidence). We are studying this model under three
assumptions on the geometry of the table and the force:

\medskip\noindent \textbf{Assumption A (additional integral).} A smooth
function $\mathcal{E}(\mathbf{q},\mathbf{p})$ is preserved by the
dynamics, $\Omega := \{\mathcal{E}(\mathbf{p},\mathbf{q}) =
\text{const} \}$ is a compact 3D manifold, and for every $\mathbf{q}$
and $\mathbf{p} \neq 0$ the ray $\{(\mathbf{q}, s\mathbf{p}),\,s>0\}$
intersects $\Omega$ in exactly one point.

Under this assumption, $\Omega$ can be parameterized by $(x,y,\theta)$
where $(x,y) \in \mathcal{D}$ is the position on the torus and $\theta$
is the angle of motion measured between $\mathbf{p}$ and the positive
$x$-axis. Equations of motion can now be given as
\[
\dot x = p \cos \theta, \quad \dot y = p \sin \theta,\quad
    \dot \theta = ph
\]
where
\[
p=\|\mathbf{p}\|>0 \quad \text{and} \quad h=(-F_1 \sin \theta
    + F_2 \cos \theta ) / p^2 .
\]
Note that $p$ does not have to be constant, but it is bounded away from
$0$ and infinity: $0<p_{\rm{min}} \leq p \leq p_{\rm{max}}<\infty$ (due
to the compactness of $\Omega$).

For a function $f$ on $\Omega$ let $f_x$, $f_y$, $f_\theta$
denote its partial derivatives, and $\|f\|_{C^2}$ the maximum
of $f$ and its first and second partial derivatives over $\Omega$.
Let $B_0 = \max \left( p_{\text{min}}^{-1}, \|p\|_{C^2},
\|h\|_{C^2} \right)$.

\medskip\noindent \textbf{Assumption B (smallness of the force).} The
force $\mathbf{F}$ and its first derivatives are small:
\[
\max(|h|,|h_x|,|h_y|,|h_\theta| ) \leq \delta_0
\]
More precisely, we require that for any given $B_*>0$ there
is a small $\delta_* = \delta_*(\mathcal{D},B_*)$ such that
all our results hold whenever $B_0<B_*$ and $\delta_0 < \delta_*$.

\medskip\noindent \textbf{Assumption C (finite horizon).} There is an
$L>0$ so that every straight line on the torus of length $L$ crosses
the interior of at least one obstacle. (Both tables A and B in
Figure~\ref{pic:tables} have finite horizon.)

A particular example of the force is \textit{Gaussian thermostat},
which was the subject of paper \cite{BCKL}, being a physically
interesting model of electrical conductance. There $\mathbf{F} \cdot
\mathbf{p}=0$, thus
$\mathcal{E}(\mathbf{p},\mathbf{q})=\frac{1}{2}\|\mathbf{p}\|^2$ is
preserved by the dynamics. For more examples see \cite[Section 2]{Ch1}.

\section{Standard notation and facts}

\emph{Flow} $\Phi^t$ acts on the \emph{phase space} $\Omega$, which is
a 3D manifold.

\emph{Collision space} $\mathcal{M} \subset \Omega$ is a set of
points where the particle undergoes a collision with $\partial
\mathcal{D}$. Now $\mathcal{M}$ can be parameterized by
$(r,\varphi)$ where $r$ is an arclength parameter along $\partial
\mathcal{D}$ and $\varphi$ is an angle between the particle's
outgoing velocity and the inward normal to $\partial \mathcal{D}$.
Note that $-\pi/2 \leq \varphi \leq \pi/2$ (see \cite{Ch1}), thus
$\cM$ can be identified with a finite union of cylinders
$\cup_{i=1}^k \partial \mathcal{B}_i \times [-\pi/2,\pi/2]$, hence
the collision space is independent of the force $\mathbf{F}$.

\emph{Collision map} $\mathcal{F}:\mathcal{M} \to \mathcal{M}$ is
the natural first return map on $\mathcal{M}$. It preserves a unique
SRB measure $\nu$; see \cite{Ch1}. We denote the time between
collisions by $\tau: \mathcal{M} \to \mathbb{R}$. Now $\Phi^t$ can
be represented as a suspension flow with base $\mathcal{M}$ and the
ceiling function $\tau$. The flow $\Phi^t$ preserves a unique SRB
measure $\mu$; see \cite{Ch2}. The map $\mathcal{F}$ and the flow
$\Phi^t$ are ergodic, mixing, and Bernoulli, they enjoy strong
statistical properties \cite{Ch2}.

We will use subscript ``0'' in $\nu_0$, $\mathcal{F}_0$, $\mu_0$,
$\Phi_0^t$ etc.\ to refer to the unperturbed (billiard) dynamics on
$\cD$, i.e., to the case $\mathbf{F}=0$.

There is a simple relation between $\mu$ and $\nu$: if $F:\Omega \to
\mathbb{R}$ is a bounded function such that $f(X) = \int_0^{\tau(X)}
F(\Phi^t(X))\,dt$, then
\begin{equation}\label{munu}
    \mu(F) = \frac{\nu(f)}{\nu(\tau)} \text{.}
\end{equation}

In addition to natural singularities of $\mathcal{F}$ (the preimages
of grazing collisions characterized by $\varphi=\pm \pi/2$) we need
to cut $\mathcal{M}$ into countably many \textit{homogeneity strips}
along the lines $\{\varphi = \pm (\pi/2-k^{-2})\}$ for all $k\geq
k_0$, forcing $\mathcal{F}$ to be discontinuous on the preimages of
these lines as well; see \cite{Ch1}.

Collision space $\mathcal{M}$ has a measurable partition into
\textit{homogeneous unstable manifolds} (or h-fibers) that are
increasing curves in the $(r, \varphi)$ coordinates with slopes
uniformly bounded away from $0$ and $\infty$ and uniformly bounded
curvature. H-fibers end on \textit{singularity curves} that are
images of the lines $\{ \varphi=\pm \pi/2 \}$ and boundaries of the
homogeneity strips. It is important for us that the singularity
curves are nondecreasing in $(r,\varphi)$ coordinates and there are
countably many of them. For almost every point $X \in \mathcal{M}$
(with respect to both the Lebesgue measure on $\mathcal{M}$ and the
SRB measure $\nu$) there exists an h-fiber $\gamma = \gamma(X)$ that
contains $X$. The SRB measure $\nu$ on $\mathcal{M}$ may be singular
with respect to the Lebesgue measure, but its conditional
distributions on h-fibers are absolutely continuous with respect to
the arclength measure.

For $X,Y \in \mathcal{M}$ we define the future separation time
$\mathbf{s}_+(X,Y)$ as the first $n\geq 0$ for which $\mathcal{F}^n(X)$
and $\mathcal{F}^n(Y)$ belong to different connected components of
$\mathcal{M}$. Similarly, $\mathbf{s}_-(X,Y)$ is the first $n\geq 0$
for which $\mathcal{F}^{-n}(X)$ and $\mathcal{F}^{-n}(Y)$ belong to
different connected components of $\mathcal{M}$.

A function $f:\mathcal{M} \to \mathbb{R}$ is \textit{dynamically
H\"older continuous} if there are $0<\theta_f<1$ and $C_f>0$ such that
for any $X$ and $Y$ lying on one unstable curve\footnote{A curve is
\emph{unstable} if its tangent vectors belong to unstable cones
\cite[Sect.~4.5]{CM06}, i.e.\ $C_1< d\varphi / dr < C_2$ for some positive
constants $C_1<C_2$. Note that unstable curves are defined on $\cM$ before it is cut into connected components, i.e., they can cross singularity lines and borders of the homogeneity strips.}
\[
    \left| f(X)-f(Y) \right| \leq C_f \theta_f^{\mathbf{s}_+(X,Y)}
\]
and for any $X$ and $Y$ lying on the same stable curve
\[
    \left| f(X)-f(Y) \right| \leq C_f \theta_f^{\mathbf{s}_-(X,Y)}
    \text{.}
\]
Dynamical H\"older continuity implies boundedness of $f$.
The class of dynamically H\"older continuous functions is large, for
example it includes all piecewise H\"older-continuous functions
whose discontinuities coincide with those of $\mathcal{F}^{\pm m}$
for some $m\geq0$.

We will say that a function $\rho$ is \textit{regular} on an unstable
curve $\gamma$ if
\begin{equation}\label{eq:regular}
\left|
\ln \rho(X) - \ln \rho(Y)
\right| \leq C_r \theta_r^{\mathbf{s}_+(X,Y)}
\end{equation}
where $\theta_r=\Lambda^{-1/6}<1$ and $C_r$ is a sufficiently large
constant that is determined by the geometry of the table and can be
chosen arbitrarily high.

A \textit{standard pair} is $(\gamma, \nu_\gamma)$ is an unstable curve
$\gamma$ with a probability measure $\nu_\gamma$ on it which has a
regular density with respect to the arclength measure. More generally,
a \textit{standard family} is an arbitrary collection $\mathcal{G} =
(\gamma_\alpha, \nu_\alpha)$, $\alpha \in \mathfrak{A}$, of standard
pairs with a probability factor measure $\lambda_\mathcal{G}$ on the
index set $\mathfrak{A}$ (one can naturally define a metric on the
space of all standard pairs, see \cite[Proposition~8.1]{CD09}, then
$\mathfrak{A}$ becomes a metric space with the respective Borel
$\sigma$-algebra). Every standard family $\mathcal{G}$ naturally
induces a measure on $\mathcal{M}$ by
\[
\nu_\mathcal{G}(A) = \int_{\mathfrak{A}}
\nu_\alpha(A\cap\gamma_{\alpha})\,
    d\lambda_\mathcal{G}(\alpha) \text{.}
\]
Every point $X$ on the unstable curve $\gamma \in \cG$ breaks
it into two pieces. Denote by $r_\cG(X)$ the length of the shorter one. Let
\begin{equation}\label{eq:zg}
    \cZ_\cG = \sup_{\eps>0}
    \frac{\nu_\cG(\{r_\cG(X)<\eps\})}{\eps} \text{.}
\end{equation}
A standard family $\mathcal{G}$ is \textbf{proper} if
$\mathcal{Z}_{\mathcal{G}} < C_p$ where $C_p$ is a large but fixed
constant. A standard family consisting of all h-fibers together with
the conditional measures induced by $\nu$ is proper \cite[p.~96]{Ch2}.

\section{Regularity of projections}
We will show that despite the singularity of the SRB measure with
respect to the Lebesgue measure, its projections that are
\textit{transverse} to the unstable manifolds have continuous
densities.

The collision space $\mathcal{M}$ admits a measurable partition
$\Gamma$ into h-fibers $\gamma \subset \mathcal{M}$. The SRB measure
$\nu$ induces conditional probability measures $\nu_\gamma$ on
h-fibers $\gamma \in \Gamma$ and a factor measure $\lambda$ on
$\Gamma$ with a standard $\sigma$-algebra (see, e.g.,
\cite[p.~287]{CM06}). The measures $\nu_\gamma$ are absolutely
continuous with respect to the arclength on $\gamma$. Moreover, the
corresponding density functions $\rho_\gamma$ are $C^1$ smooth (see,
e.g., \cite[Sect.~5.2]{CM06}) and regular as defined above.

The length of h-fibers, as a function
\[ \begin{array}{ccc}
L: & \cM \to \mathbb{R} \\
   & X \mapsto |\gamma| &\, \text{for}\, \gamma \ni X
\end{array}
\]
is measurable by the dominated convergence theorem: for almost every $X
\in \cM$ let $B_n(X)$ be a connected component of the domain of
$\cF^{-n}$ which contains $X$. Let $L_n(X) = \sup \{ |\gamma|:\, \gamma
\in \Gamma,\, \gamma \subset B_n(X) \}$ be a supremum of lengths of
h-fibers in $B_n(X)$. Then $L(X) = \lim_{n \to \infty} L_n(X)$ (see the
beginning of Chapter 5 in \cite{CM06} for more information on the
structure of $\Gamma$.) Hence the length of h-fibers is also a
measurable function on $\Gamma$, by a straightforward verification.

Since $\Gamma$ is a proper standard family, $\mathcal{Z}_\Gamma <
\infty$, and therefore \cite[Sect.~7.4]{CM06}
\begin{equation} \label{intGamma}
   \int_\Gamma \frac{d\lambda(\gamma)}{|\gamma|} <\infty.
\end{equation}
This allows us to renormalize the conditional measures and the
factor measure replacing $d \nu_\gamma$ and $d \lambda (\gamma)$ as
follows:
$$
  d \hat{\nu}_\gamma = |\gamma| d \nu_\gamma
  \qquad\text{and}\qquad
  d \hat{\lambda}(\gamma) = \frac{d \lambda(\gamma)}{|\gamma|}.
$$
The new factor measure $\hat\lambda$ is still finite, according to
\eqref{intGamma}.

Since $\nu_\gamma$ is a probability measure for each $\gamma$, and
$\rho_\gamma$ is $C^1$ smooth and regular, $\rho_\gamma$ is bounded
by $|\gamma|^{-1} e^{-C_r}$ from below and by $|\gamma|^{-1}
e^{C_r}$ from above, and thus the density $\hat \rho _ \gamma$ of
the new measure $\hat{\nu}_{\gamma}$ is also $C^1$ smooth and
bounded \emph{uniformly} by $e^{-C_r}$ and $e^{C_r}$.

\begin{lemma} \label{Lm4.1}
The $\nu$-measure of every h-fiber $\gamma \in \Gamma$ is zero, i.e.\
$\lambda(\gamma) = \hat{\lambda}(\gamma)=0$.
\end{lemma}

\begin{proof}
Assume that an h-fiber $\gamma$ has positive measure. Recall that
$\rho_\gamma$ is bounded away from $0$ by $|\gamma|^{-1} e^{-C_r}$.
Note that the collision map $\mathcal{F}$ is piecewise continuous
and bijective. For every $n>0$ the curve $\mathcal{F}^{-n}(\gamma)$
is a piece of some h-fiber; it carries the same measure as $\gamma$,
but its length is $O(\Lambda^{-n} |\gamma|)$. By the Poincar\'e
recurrence theorem $\mathcal{F}^{-n} (\gamma)$ must overlap with
$\gamma$ infinitely many times, which implies that the measure of a
relatively long piece of $\gamma$ is equal to the measure of an
arbitrarily short piece, hence density $\rho_\gamma$ cannot be
bounded.
\end{proof}

In the following theorem $\cX$ denotes an abstract manifold, but
$\Gamma$ is still the partition of our collision space $\cM$ and  $\hat{\lambda}$ is still the factor measure defined above.

\begin{theorem}\label{topX}
Let $\cX$ be a compact Riemannian manifold equipped with Lebesgue
measure $dx$. Assume that for each $\gamma \in \Gamma$ there is a
function $p_{\gamma} \colon \cX \to \mathbb{R}$, which is bounded
uniformly in $\gamma$. Define a (possibly signed) measure $\xi$ on
$\cX$ by
\[
 \xi(A) = \int_\Gamma \zeta_\gamma(A) \,
 d\hat{\lambda}(\gamma),\qquad
    \zeta_\gamma(A) = \int_A p_\gamma(x) \, dx
\]
where we assume that $\zeta_\gamma(A)$ is measurable, as a function
of $\gamma$, for every measurable $A \subset \cX$. Assume that for
every point $x\in\cX$
\begin{equation} \label{E0}
   \hat{\lambda}\bigl\{\gamma\in\Gamma\colon
   p_{\gamma}\text{ is discontinuous at }x\bigr\}=0.
\end{equation}
Then the measure $\xi$ has a continuous density on $\cX$ with respect
to $dx$. If, in addition, for every $x\in\cX$
\begin{equation} \label{E1}
   \hat{\lambda}\bigl\{\gamma\in\Gamma\colon
   p_{\gamma}(x)>0 \bigr\} >0
\end{equation}
then the density of $\xi$ is strictly positive and bounded away from
zero.
\end{theorem}

\begin{proof}
For $x \in \cX$ and $r>0$ let $B_r(x) \subset \cX$ denote the ball
of radius $r$ centered at $x$ and $|B_r(x)|$ its Lebesgue volume.
Then for $\hat{\lambda}$-almost every $\gamma \in \Gamma$
\[
   \lim_{r\to 0} \frac{\zeta_\gamma(B_r(x))}{|B_r(x)|} = p_\gamma(x).
\]
By the bounded convergence theorem $p_\gamma(x)$ is a measurable
function on $\Gamma$ and
\[
    \lim_{r\to 0}
    \frac{\xi(B_r(x))}{|B_r(x)|} = p(x) := \int_\Gamma
    p_\gamma(x) \, d\hat{\lambda}(\gamma).
\]
By the Lebesgue differentiation theorem $p(x)$ is almost everywhere
on $\cX$ equal to the density of $\xi$. The continuity and
positivity of $p(x)$ under our assumptions follows directly from the
bounded convergence theorem.
\end{proof}

Next we show how Theorem~\ref{topX} implies the continuity of
various projections of the SRB measure $\nu$. In all our cases,
every $x \in\cX$ will be a discontinuity point for at most countably
many functions $p_\gamma(x)$. This, along with Lemma~\ref{Lm4.1},
will guarantee the assumption \eqref{E0}.

\subsection{Angular distribution for the collision map}

Let $\cX=[-\pi/2, \pi/2]$ and $P\colon \cM \to \cX$ be the projection
onto the $\varphi$-coordinate. Let $\xi$ be the corresponding
pushforward of $\nu$.

\begin{theorem}  \label{TmAD}
The measure $\xi$ is absolutely continuous on $\cX=[-\pi/2, \pi/2]$
with a positive continuous density.
\end{theorem}

\begin{proof} For $A \subset \cX$ we have $\xi(A) = \nu(P^{-1}(A))$ and
\[\xi(A) = \int_\Gamma \zeta_\gamma(A) \,
d\hat{\lambda}(\gamma), \qquad \zeta_\gamma(A) =
\hat{\nu}_\gamma\left(\gamma \cap P^{-1}(A)\right)\] Since the
h-fibers are increasing curves in the $(r, \varphi)$ coordinates
with slopes uniformly bounded away from $0$ and $\infty$ and the
densities $\hat{\rho}_{\gamma}$ of the respective measures
$\hat{\nu}_\gamma$ are uniformly bounded, all the measures
$\zeta_\gamma$ have uniformly bounded piecewise continuous densities
on $\cX$ (the densities of $\zeta_\gamma$ correspond to $p_\gamma$
in Theorem~\ref{topX}).

Recall that the h-fibers terminate on singularity curves, and every
line $\varphi=\text{const}$ intersects at most countably many
singularities. Thus at most countably many unstable manifolds
terminate on that line, hence at most countably many projected
densities have discontinuities at any given $\varphi \in \cX$. This
guarantees \eqref{E0}, and \eqref{E1} follows from \cite[Lemma
3.3]{Ch2} because the preimage of each line $\varphi=\text{const}$
is a finite union of stable curves. For every measurable $A \subset
\cX$, the measurability of the function $\gamma \mapsto
\zeta_\gamma(A)$ follows from the measurability of the partition
$\Gamma$. Now the result follows from Theorem~\ref{topX}.
\end{proof}

Exactly the same argument shows that projection onto the $r$
coordinate has a positive continuous density.

\subsection{Projection on $\mathcal{D}$}
We leave the collision space $\mathcal{M}$ and project the SRB measure
$\mu$ for the flow $\Phi^t$ onto the table $\mathcal{D}$. Let $P\colon
\Omega \to \cD$ be the projection of the phase space onto the
configuration space, and $\xi$ be the pushforward of $\mu$.

\begin{theorem}
The measure $\xi$ is absolutely continuous on $\mathcal{D}$ with a
positive continuous density.
\end{theorem}

\begin{proof}
According to formula (\ref{munu}), for any set $A \subset \cD$
\[
    \xi(A) = \mu\bigl(P^{-1}(A)\bigr)
    = \frac{\nu(f_A)}{\nu(\tau)}\text{\,\, with}
\]
\[
    f_A(X) = \int_0^{\tau(X)} \mathbf{1}_A(P(\Phi^t(X)))\,dt.
\]
Therefore
\[
    \xi(A) = \frac{1}{\nu(\tau)}
    \int_\Gamma \hat{\nu}_\gamma (f_A)\, d\hat{\lambda}(\gamma).
\]
We observe that the map $A \mapsto \hat{\nu}_\gamma(f_A)$ defines a
measure on $\cD$ with a piecewise continuous density. This measure is
supported on the \textit{trace} of h-fiber $\gamma$ (see
Figure~\ref{unstable-table}), on which the density is positive and
continuous. This density corresponds to $p_\gamma$ in
Theorem~\ref{topX}. The densities $p_\gamma$ on $\mathcal{D}$ are
bounded uniformly in $\gamma$ because the density of $\hat{\nu}_\gamma$
on each $\gamma$ is uniformly bounded.

\begin{figure}[h!]
\centering
\includegraphics{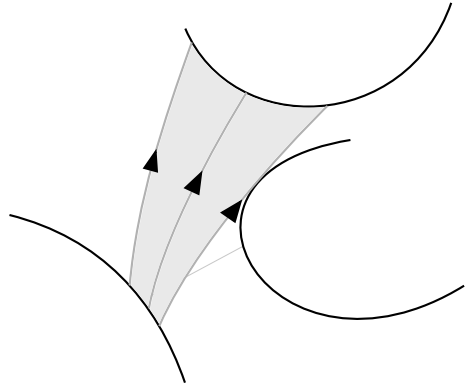}
\caption{$\hat{\nu}_\gamma(f_A)$ is supported on trajectories that start
from an h-fiber $\gamma$}
\label{unstable-table}
\end{figure}

Furthermore, for any $x \in \cD$ define the set
\[E_x:=\left\{ X \in \cM:\,\,
\{\Phi^t(X)\}_{t=0}^{\tau(X)} \cap P^{-1}(x) \text{ is not empty}
\right\}\] of points in the collision space. The trajectories
starting from $E_x$ pass through $x$ and therefore correspond to a
focusing wave front \cite[Sect.~3.7]{CM06}. In other words, $E_x$
consists of a finite number of decreasing curves in the
$(r,\varphi)$-coordinates, which have countably many intersections
with singularity curves on which h-fibers terminate. Recall that
h-fibers also terminate on the preimages of lines $\{\varphi=\pm
\pi/2\}$. It is clear that if $x \in \mathcal{D}$ is a point of
discontinuity for the density $p_\gamma$ of some h-fiber $\gamma$,
then $\gamma$ must terminate on $E_x$. Therefore $x$ can be a point
of discontinuity for at most countably many densities $p_\gamma$.
This guarantees \eqref{E0}, and \eqref{E1} again follows from
\cite[Lemma 3.3]{Ch2} because $E_x$ is a finite union of stable
curves. Now the result follows from Theorem~\ref{topX}.
\end{proof}

Note that the density of the measure $\xi$ on $\mathcal{D}$ is given by
\[
 p(x) := \frac{1}{\nu(\tau)} \int_\Gamma p_\gamma(x)\, d\hat{\lambda}(\gamma).
\]

Next we show that the velocity field is continuous in the following
sense. For every point $\mathbf{q}\in \cD$ let
$\bar{\mathbf{p}}(\mathbf{q}) = \mu_{\mathbf{q}}(\mathbf{p})$ denote
the average velocity vector, where $\mu_{\mathbf{q}}$ denotes the
conditional measure induced by $\mu$ on the section of the phase
space $\Omega$ corresponding to the fixed footpoint $\mathbf{q}$.

\begin{theorem}
The velocity vector field $\bar{\mathbf{p}}(\mathbf{q})$ is continuous
on $\cD$.
\end{theorem}

\begin{proof}
In place of the SRB measure for the flow $\mu$ we use a signed
measure $\mu_1$, defined by $\mu_1(F) = \mu(v_1 F)$ for any function
$F$ on $\Omega$, where $v_1(X)$ is the horizontal component of the
velocity of the particle at $X \in \Omega$. Then the projection of
$\mu_1$ onto $\mathcal{D}$ is a signed measure $\xi_1$ defined by
\[
 \xi_1(A)=\mu_1 (P^{-1}(A)) = \mu \left(v_1 \mathbf{1}_{P^{-1}(A)} \right) =
 \frac{\nu(f_{1,A})}{\nu(\tau)} \text{, where}
\]
\[
 f_{1,A}(X) = \int_0^{\tau(X)} \left(v_1 \mathbf{1}_{P^{-1}(A)} \right) \left(
 \Phi^t(X)\right) dt,
\]
and
\[
 \xi_1(A) = \frac{1}{\nu(\tau)} \int_\Gamma \hat{\nu}(f_{1,A})\,d\hat{\lambda}(\gamma).
\]
Once again, the map $A \mapsto \hat{\nu}(f_{1,A})$ defines a (signed)
measure on $\mathcal{D}$ with density $p_{1,\gamma}$ that has the same
properties as $p_\gamma$ above, so that Theorem~\ref{topX} applies to
prove that $\xi_1$ has continuous a density on $\mathcal{D}$ given by
\[
 p_1(\mathbf{q}) := \frac{1}{\nu(\tau)}
 \int_\Gamma p_{1,\gamma}(\mathbf{q})\, d\hat{\lambda}(\gamma)
\]

Note that the average horizontal velocity on the set $A \subset
\mathcal{D}$ is given by $\mu\left( v_1 \mathbf{1}_{P^{-1}(A)} \right)
/ \mu \left( \mathbf{1}_{P^{-1}(A)} \right) =\xi_1(A)/\xi(A)$. Taking
into account that $\xi$ and $\xi_1$ have continuous densities $p$ and
$p_1$, and $p>0$ everywhere, the average horizontal velocity at every
point $X \in \mathcal{D}$ is well defined as
$p_1(\mathbf{q})/p(\mathbf{q})$ and is continuous. The same argument
works for the vertical component of the velocity.
\end{proof}

Figure~\ref{pic:velocity} shows the computed velocity field in a system
with a Gaussian thermostated force $\mathbf{F}$ directed horizontally
to the right.

\begin{figure}[h!]
\centering
\includegraphics[clip,width=0.48\textwidth]{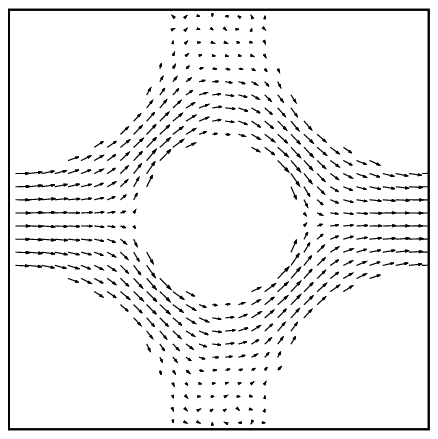}
\includegraphics[clip,width=0.48\textwidth]{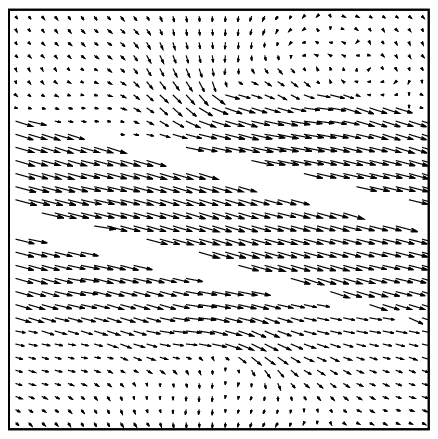}
\caption{Velocity fields on tables A and B (see
Figure~\ref{pic:tables}).\label{pic:velocity}}
\end{figure}

\subsection{Angular distribution for the flow with Gaussian thermostat}
\label{sec:gt}

This section is restricted to a specific model --- a constant external
field with Gaussian thermostat. The force is given by $\bF = \bE -
\frac{\bE \cdot \bp}{|\bp|^2}$, where $\bE$ is a small nonzero constant
vector. This model was the subject of our paper \cite{BCKL}.

We can choose the coordinate system so that $\bE$ points in the
positive $x$ direction, then $\theta\in [-\pi,\pi]$ (see Assumption A)
measures the angle between the particle velocity and the field $\bE$.
It is a direct verification that the only straight trajectories (where
$\dot \theta \neq 0$) are those parallel to the field $\bE$, i.e.,
$\theta\in \{0,\pm \pi\}$. This fact is used in the construction below.

Let $\cX = [-\pi,\pi]$, denote by $P \colon \Omega \to \cX$ the
projection of the phase space onto the $\theta$ coordinate, and by
$\xi$ the pushforward of $\mu$.

\begin{theorem}
The measure $\xi$ is absolutely continuous on $\cX$ with a
positive continuous density.
\end{theorem}

\begin{proof}
For every set $A \subset \Omega$ we have:
\[
  \mu(A) = \frac{1}{\nu(\tau)}
    \int_\Gamma \hat{\nu}_\gamma (f_A)\, d\hat{\lambda}(\gamma), \, \text{ where}
\]
\[
    f_A(X) = \int_0^{\tau(X)} \mathbf{1}_A(\Phi^t(X))\,dt.
\]
Denote $\mu_\gamma(A) := \hat{\nu}_\gamma (f_A)$ and let
$\hat{\mu}_\gamma$ be the projection of $\mu_\gamma$ on $\cX$. Then for
any set $B \subset \cX$
\[
 \xi(B) = \frac{1}{\nu(\tau)}
    \int_\Gamma \hat{\mu}_\gamma (B)\, d\hat{\lambda}(\gamma).
\]
Let $S_\gamma := \{ \Phi^t(X) \colon X \in \gamma \text{ and } 0\leq t
\leq \tau(X) \}$ be again the \textit{trace} of $\gamma$. Then
$\mu_\gamma(A)$ is supported on $S_\gamma$ and has continuous and
uniformly (in $\gamma$) bounded density on it.

Since $S_\gamma$ is a compact smooth 2-dimensional  manifold in the
phase space $\Omega$ and h-fibers correspond to \textit{strongly
divergent families of trajectories}\footnote{ Let $\kappa$ denote the
curvature of a cross-section of $S_\gamma$ orthogonal to the flow; then
\textit{strong divergence} means $0< \kappa_{\rm{min}} < \kappa <
\kappa_{\rm{max}} < \infty $ for some global constants
$\kappa_{\rm{min}}$ and $\kappa_{\rm{max}}$, see \cite[Secion 3]{Ch1}.
}, the angle between $S_\gamma$ and the $\theta$-axis is bounded above
by a global constant that is less than $\pi/2$. Therefore,
$\hat{\mu}_\gamma$ has a density on $\cX$. Moreover, the area and the
size (diameter) of $S_\gamma$ are uniformly (in $\gamma$) bounded above
and below by positive constants. Thus the density of $\hat{\mu}_\gamma$
is bounded above uniformly in $\gamma$.

The density of $\hat{\mu}_\gamma$ may have a discontinuity at $\theta_0 \in
\cX$ only if $S_\gamma$ has a curve on its boundary where $\theta
\equiv \theta_0$. The boundary of $S_\gamma$ consists of four parts:
\begin{align*}
  S_\gamma^1 & := \{ \Phi^t(X) \colon X \in \gamma \text{ and } t = 0 \}, \\
  S_\gamma^2 & := \{ \Phi^t(X) \colon X \in \gamma \text{ and } t = \tau(X) \}, \\
  S_\gamma^3 & := \{ \Phi^t(X) \colon X=X_1 \text{ and } 0 \leq t \leq \tau(X) \}, \\
  S_\gamma^4 & := \{ \Phi^t(X) \colon X=X_2 \text{ and } 0 \leq t \leq \tau(X) \},
\end{align*}
where $X_1$ and $X_2$ are the endpoints of $\gamma$.

The angle $\theta$ cannot be constant on any sub-curve of $S_\gamma^1$,
because $\gamma$ is an increasing curve in the $(r,\varphi)$
coordinates. On $S_\gamma^3$ or $S_\gamma^4$, it can be a constant only
if the whole trajectory is parallel to $\bE$, i.e., $\theta\in \{0,\pm
\pi\}$. The trajectories where $\theta\equiv 0$ or $\theta \equiv \pm
\pi$ make a finite union of flat wave fronts, with at most countably
many intersections with singularity curves and preimages of $\{\varphi
= \pm \pi/2\}$, on which h-fibers terminate. Hence there are at most
countably many $\gamma$'s for which $S_\gamma^3$ or $S_\gamma^4$ has a
sub-curve where $\theta\equiv 0$ or $\theta \equiv \pm \pi$.

It may happen that $S_\gamma^2$ contains a subset of positive length on
which $\theta\equiv \,$const. Note, however, that for every $\theta_0
\in \cX$ the set
$$
  H_{\theta_0} \colon = \{ X \in \cM \colon P(\Phi^{0-}(X))=\theta_0 \}
$$
of reflection points where the ``incoming'' velocity vector makes angle
$\theta_0$ with the field $\bE$, is a finite union of smooth curves, so
its preimage $\cF^{-1}(H_{\theta_0})$ is a finite union of smooth
curves, too. It is now clear that there could be at most countably many
$\gamma$'s that partially coincide with $\cF^{-1}(H_{\theta_0})$.

Overall, for every $\theta_0 \in \cX$ there are at most countably many
$\gamma$'s for which the boundary of $S_\gamma$ contains a curve on
which $\theta \equiv \theta_0$. This guarantees (\ref{E0}).

To satisfy (\ref{E1}), for every $\theta_0 \in \cX$ consider a set
$$
  E_{\theta_0} \colon= \{ X \in \cM \colon P(\Phi^{0+}(X)) = \theta_0 \}
$$
of reflection points where the ``outgoing'' velocity vector makes angle
$\theta_0$ with the field $\bE$. It consists of a finite number of
decreasing curves, and its preimage $\cF^{-1} (E_{\theta_0})$ consists
of a finite number of stable curves. Consider two sets
\[
   \Gamma_{\theta_0} \colon= \{\gamma \colon \gamma \text{ intersects } E_{\theta_0} \}
   \text{ \, and \, }
   \Gamma_{\theta_0}' \colon= \{\gamma \colon \gamma \text{ terminates on } E_{\theta_0} \}.
\]
It follows from \cite[Lemma 3.3]{Ch2} that $\Gamma_{\theta_0}$ has
positive $\hat{\lambda}$-measure in $\Gamma$, and $\Gamma_{\theta_0}'$
is at most countable, because it has at most countably many
intersection points with singularity curves and preimages of $\{
\varphi = \pm \pi/2 \}$, on which h-fibers terminate. Therefore, $\hat
\lambda (\Gamma_{\theta_0} \setminus \Gamma_{\theta_0}') > 0$. Observe
that for every $\gamma \in \Gamma_{\theta_0} \setminus
\Gamma_{\theta_0}'$ the density of the projected measure
$\hat{\mu}_\gamma$ is positive at $\theta_0$. This implies (\ref{E1}).

Now the result follows from Theorem~\ref{topX}.
\end{proof}

Figure~\ref{pic:angular} shows the density of $\xi$ constructed via
computer simulation for a system under a small external force directed
horizontally to the right, with Gaussian thermostat.

\begin{figure}[h!]
\centering
\includegraphics[clip,width=0.48\textwidth]{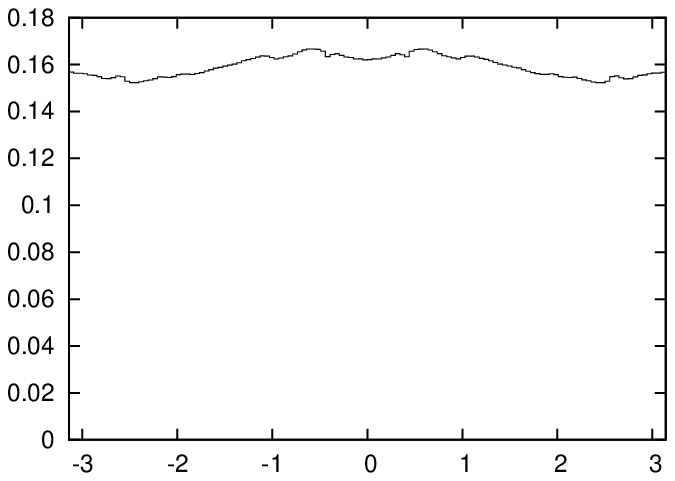}
\includegraphics[clip,width=0.48\textwidth]{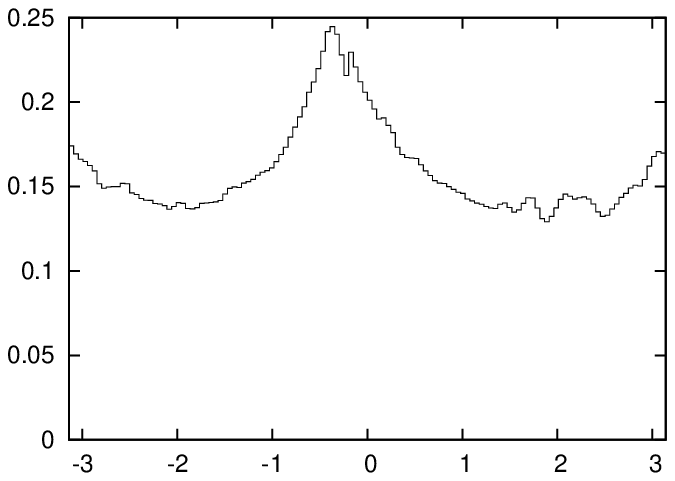}
\caption{
Densities of angular distributions for a Gaussian thermostated force
with the same small field on the tables A and B
(see Figure~\ref{pic:tables}).
\label{pic:angular}}
\end{figure}

\section{Linear response}

Suppose that the force $\mathbf{F}=\mathbf{F}_\eps$ is parameterized
by a parameter $\eps \in [0, \bar{\eps}]$. More precisely,
Assumption~B in section~\ref{sec:model} now takes form
\[
  \max(|h|,|h_x|,|h_y|,|h_\theta| ) \leq C_0\varepsilon
\]
with some $C_0>0$ and $B_0>0$ independent of $\varepsilon$. We will
add the subscript $\eps$ to our symbols to emphasize the dependence
of the dynamics on $\eps$. Let $g_\eps=d\mathcal{F}_\eps^{-1}\nu_0 /
d \nu_0$ be the Jacobian of $\mathcal{F}_\eps$ with respect to the
unperturbed billiard invariant measure $\nu_0$. Denote $\Delta_\eps
\colon= (1-g_\eps)/\eps$ for $\eps>0$ and assume that $\Delta_0
\colon=\lim_{\eps \to 0} (1-g_\eps)/\eps$ exists almost everywhere
with respect to $\nu_0$.

We will make a rather technical, but not too restrictive assumption
on functions $\Delta_\eps$, namely that each h-fiber of the map
$\mathcal{F}_0$ can be divided into no more than $N_\Delta$ pieces
(for some constant $N_\Delta$), on which they are H\"older continuous
with a constant $C_\Delta>0$ and exponent $1/6$,
i.e.\ $|\Delta_\eps(X) - \Delta_\eps(Y)| \leq C_\Delta |X-Y|^{1/6}$,
and that $|\Delta_\eps|<C_\Delta$. We need exponent $1/6$ to connect
$\Delta_\eps$ to a proper standard family, as it will be shown
further.

For example, for the Gaussian thermostat in Section \ref{sec:gt},
the function $\Delta_\eps$ satisfies our assumptions with H\"older exponent $1/2$, and constant $C_\Delta$ determined by the maximum and minimum curvature of the obstacles; see \cite[Sect.~8]{CD09a}.

Assume that we are observing a function $f_\eps$ on $\mathcal{M}$
that may also change with $\eps$. Let $f_\eps$ be bounded uniformly
in $\eps$ and dynamically H\"older continuous with respect to
$\mathcal{F}_\eps$, with constants $C'=C_{f_\eps}$ and
$\theta'=\theta_{f_\eps}$ independent from $\eps$. Suppose the limit
$f_0:=\lim_{\eps \to 0} f_\eps$ exists almost everywhere with
respect to $\nu_0$.

\begin{theorem}
\begin{equation}\label{linResp}
\nu_\eps(f_\eps) - \nu_0(f_\eps) =
\eps \sum_{k=1}^\infty \nu_0
\left( (f_0 \circ \mathcal{F}_0^k)\Delta_0 \right) + o(\eps).
\end{equation}
\end{theorem}
\begin{remark}
 It may not be true that $\nu_0(f_\eps) = \nu_0(f_0) + C\eps + o(\eps)$,
 this part of response is determined by the character of $f_\eps$.
\end{remark}

\begin{proof}
We start with a Kawasaki-type formula \cite[Eq.~(2.15)]{Ch2}:
\[
\nu_\eps(f_\eps) = \nu_0(f_\eps) + \sum_{k=1}^\infty \eps \, \nu_0
\left[\left( f_\eps \circ \mathcal{F}_\eps^k \right)\Delta_\eps\right].
\]
Here terms of the series decay exponentially fast and uniformly in
$\eps$. We are going to factor $\eps$ out of the series and prove
that we still have a series with terms converging to zero
exponentially and uniformly in $\eps$.

Note that the integrands $\left( f_\eps \circ \mathcal{F}_\eps^k
\right)\Delta_\eps$ are bounded and pointwise converge to $\left( f_0
\circ \mathcal{F}_0^k \right)\Delta_0$, as $\eps\to 0$. Therefore it is
enough to show that there exist constants $C>0$ and $0<\theta<1$ so
that $\left|\nu_0\left[ \left( f_\eps \circ \mathcal{F}_\eps^k
\right)\Delta_\eps\right] \right| \leq C \theta^k$ for every $k$.

We decompose $\Delta_\eps$ into $\Delta_\eps^+ - \Delta_\eps^-$ where
$\Delta_\eps^+=1+\Delta_\eps \mathbf{1}_{\Delta_\eps>0}$ and
$\Delta_\eps^-=1-\Delta_\eps \mathbf{1}_{\Delta_\eps<0}$. Then
$\nu_0(\Delta_\eps^+)=\nu_0(\Delta_\eps^-)$ since
$\nu_0(\Delta_\eps)=0$.

Let $\gamma$ be a piece of an h-fiber corresponding to the unperturbed
dynamics, on which $\Delta_\eps$ is H\"older continuous with constant
$C_\Delta$ and exponent $1/6$,
and $\gamma(X,Y)$ be its subcurve that terminates at points
$X$ and $Y$. Then
\begin{align*}
 \left|\ln \frac{\Delta_\eps^{\pm}(X)}{\Delta_\eps^{\pm}(Y)}\right|
 &=\left| \ln \left( 1+
   \frac{\Delta_\eps^{\pm}(X)-\Delta_\eps^{\pm}(Y)}{\Delta_\eps^{\pm}(Y)}
   \right) \right| \\
 &\leq \left|\Delta_\eps^{\pm}(X)-\Delta_\eps^{\pm}(Y)\right|
 \leq C_\Delta |\gamma(X,Y)|^{1/6} \\
 &\leq C_\Delta C_1 \Lambda^{-\frac{1}{6} s_+(X,Y)}
 =C_\Delta C_1 \theta_r^{s_+(X,Y)}
\end{align*}
where constant $C_1$ is determined by the billiard
geometry, see \cite[Formula (5.32)]{CM06}, and
$\theta_r$ comes from (\ref{eq:regular}).

If $\rho_{\gamma}$ is the density of the conditional measure induced
by $\nu_0$ on $\gamma$, then
\[
 \left| \ln \frac{\rho_{\gamma}(X)\Delta_\eps^\pm(X)}{\rho_{\gamma}(Y)\Delta_\eps^\pm(Y)}
 \right|
 \leq (C_r + C_\Delta C_1) \theta_r^{s_+(X,Y)},
\]
where constant $C_r$ also comes from (\ref{eq:regular}). We are free to
choose $C_r$ arbitrarily large, and replacing $C_r$ with $C_r +
C_\Delta C_1$ we make $\rho_{\gamma} \Delta_\eps^\pm$ a regular
function on $\gamma$.

Thus the densities $\rho_{\gamma} \Delta_\eps^\pm$ on the pieces of
h-fibers, where $\Delta_\eps$ is continuous, specify standard
families $\cG_\eps^\pm$ such that
$\nu_{\cG_\eps^\pm}(h) = \nu_0(\Delta_\eps^\pm h)
/\nu_0(\Delta_\eps^\pm)$ for every function $h$.

The standard family $\cG$ consisting of all h-fibers with their conditional measures corresponding to the unperturbed dynamics, is proper, i.e., 
$\cZ_\cG<C_p$; see equation (\ref{eq:zg}). We want to show that
$\cG_\eps^\pm$ is also proper.

To each standard pair $(\gamma, \nu_\gamma) \in \cG$ there correspond
at most $N_\Delta$ standard pairs $(\gamma_i, \nu_{\gamma_i})$
in $\cG_\eps^\pm$ with regular densities. Then, in the notations of
equation (\ref{eq:zg}),
\[
  \frac{\sum_i \nu_{\gamma_i}
    \left\{ X \colon r_{\cG_\eps^\pm} (X)<\eps \right\}
    }{\eps}
  \leq
    N_\Delta e^{C_r}
    \frac{\nu_{\gamma}
    \left\{ X \colon r_{\cG} (X)<\eps \right\}
    }{\eps}.
\]
Therefore $\cZ_{\cG_\eps^\pm} \leq N_\Delta e^{C_r} \cZ_\cG \leq N_\Delta e^{C_r} C_p$.
We are now free to choose $C_p$ large enough to make the standard families $\cG_\eps^\pm$ proper.

By the equidistribution property \cite[Proposition 2.2]{Ch2}
\[
    \left|
    \nu_{\mathcal{G}_\eps^\pm}\left(f_\eps \circ \mathcal{F}_\eps^k
    \right) - \nu_\eps\left(f_\eps\right)
    \right| \leq C_8 \theta_8^k \text{,}
\]
where $C_8>0$ and $0<\theta_8<1$ are independent from $\eps$.
Coupled with
\[
    \nu_0 \left[\left(f_\eps \circ \mathcal{F}_\eps^k\right)
    \Delta_\eps
    \right] =
    \nu_0\left(\Delta_\eps^+\right) \left(
    \nu_{\mathcal{G}_\eps^+}
        \left(f_\eps \circ \mathcal{F}_\eps^k\right)
    -
    \nu_{\mathcal{G}_\eps^-}
        \left(f_\eps \circ \mathcal{F}_\eps^k\right)
    \right)
\]
this gives
\[\nu_0 \left[\left(f_\eps \circ \mathcal{F}_\eps^k\right)
    \Delta_\eps
    \right]
    \leq
    2 \nu_0\left(\Delta_\eps^+\right)C_8 \theta_8^k \text{.}
\]
We note that $\nu_0\left(\Delta_\eps^+\right)\leq1+C_\Delta<\infty$,
which completes the proof.
\end{proof}

We proved the linear response formula for the map $\mathcal{F}_\eps$
but not for the flow, because no estimates on correlations for
perturbed billiard flows are available.

\section*{Acknowledgement}
The authors are partially supported by NSF grant DMS-0969187. We are
grateful to the Alabama supercomputer administration for computational
resources. We are grateful to the anonymous reviewers for the very
thorough job they have done. The paper significantly improved due to
their suggestions.


\begin{thebibliography}{99}

\bibitem{BCKL} F.~Bonetto, N.~Chernov, A.~Korepanov, and J.~Lebowitz,
    \emph{Spatial Structure of Stationary Nonequilibrium States in the
    Thermostatted Periodic Lorentz Gas},
    J. Statist. Phys. \textbf{146} (2012), 1221--1243.

\bibitem{BDL} F.~Bonetto, D.~Daem, J.~L.~Lebowitz,
    \emph{Properties of Stationary Nonequilibrium States in the
    Thermostatted Periodic Lorentz Gas I: The One Particle System},
    J. Statist. Phys. \textbf{101} (2000), 35--60.

\bibitem{Ch1} N. Chernov, \emph{Sinai billiards under small external
    forces}, Ann. H. Poincar\'e \textbf{2} (2001), 197--236.

\bibitem{Ch2} N. Chernov, \emph{Sinai billiards under small external
    forces II}, Ann. H. Poincar\'e \textbf{9} (2008), 91--107.

\bibitem{CD09} N. Chernov and D. Dolgopyat, \emph{Lorentz gas with
    thermostatted walls}, Ann. H. Poincar\'e \textbf{11} (2010), 1117--1169.

\bibitem{CD09a} N. Chernov and D. Dolgopyat,  \emph{Anomalous
    current in periodic Lorentz gases with infinite horizon}, Russ. Math. Surv. \textbf{64} (2009), 651--699.

\bibitem{CELS} N. Chernov, G.L. Eyink, J.L. Lebowitz, Ya.G. Sinai,
    \emph{Derivation of Ohm's law in a deterministic mechanical model},
    Phys. Rev. Lett., \textbf{70}:15 (1993), 2209--2212.

\bibitem{CELSp} N. Chernov, G.L. Eyink, J.L. Lebowitz, Ya.G. Sinai,
    \emph{Steady-state electrical conduction in the periodic Lorentz gas},
    Commun. Math. Phys., \textbf{154} (1993),
    569--601.

\bibitem{CM06}
    N. Chernov and R. Markarian, \emph{Chaotic Billiards},
    Math. Surv. Monogr., \textbf{127}, AMS, Providence,
    RI, 2006. (316 pp.)

\bibitem{GC} G. Gallavotti and E. Cohen, \emph{Dynamical ensembles
    in stationary states}, J. Statist. Phys. \textbf{80} (1995), 931--970.

\bibitem{H} B. Hasselblatt, \emph{Hyperbolic dynamical systems},
    Handbook of Dynamical Systems, Volume 1A, Ed. by B. Hasselblatt
    and A. Katok, Elsevier, Amsterdam, 2002.

\bibitem{LL}
    F.~Ledrappier, E.~Lindenstrauss, \emph{On the projections
    of measures invariant under the geodesic flow},
    Int. Math. Res. Not. 2003, no. 9, 511--526.

\bibitem{MH}B.~Moran, W.~G.~Hoover, \emph{Diffusion in the periodic
    Lorentz billiard}, J. Statist. Phys. \textbf{48} (1987), 709--726.

\bibitem{Y} L.-S. Young, \emph{What are SRB measures, and which
dynamical systems have them?}, J. Statist. Phys. \textbf{108}
(2002), 733--751.

\bibitem{Zh11} H.~K.~Zhang, \emph{Current in Periodic Lorentz Gases
    with Twists}, Commun. Math. Phys. \textbf{306}
   (2011), 747--776.

\end{thebibliography}
\end{document}